\documentclass[a4paper]{amsart}
\usepackage{soul}
\usepackage[english]{babel}
\usepackage{amssymb}
\usepackage{enumerate}
\usepackage{hyperref}
\usepackage{tikz}
\usepackage[margin=10pt,font=small,labelfont=bf,labelsep=endash]{caption}
\usetikzlibrary{matrix}
\usepackage{verbatim}

\numberwithin{equation}{section}
\theoremstyle{plain}
\newtheorem{lemma}{Lemma}[section]
\newtheorem{theorem}[lemma]{Theorem}
\newtheorem*{theorem*}{Theorem}
\newtheorem{proposition}[lemma]{Proposition}
\newtheorem{corollary}[lemma]{Corollary}

\theoremstyle{definition}
\newtheorem{definition}[lemma]{Definition}
\newtheorem*{definition*}{Definition}

\theoremstyle{remark}
\newtheorem{remark}[lemma]{Remark}

\DeclareMathOperator{\Span}{span}
\DeclareMathOperator{\card}{card}
\DeclareMathOperator{\Ad}{Ad}

\newcommand{\Wab}{\widehat W_{\mathrm{ab}}}

\newcommand{\calA}{\mathcal A} 
\newcommand{\calF}{\mathcal F}

\newcommand\half{\frac{1}{2}}

\newcommand{\mN}{\mathbb N}
\newcommand{\mQ}{\mathbb Q}
\newcommand{\mZ}{\mathbb Z}

\newcommand\Dp{\Phi^+}
\newcommand\Dap{\widehat\Phi^+}
\newcommand\Wa{\widehat{W}}

\renewcommand\b{\mathfrak b}
\renewcommand\d{\delta}
\renewcommand\t{\mathfrak t}
\renewcommand\u{\mathfrak u}

\renewcommand\a{\alpha}
\renewcommand\aa{\mathfrak a}
\newcommand\goa{\mathfrak a}
\newcommand\wgoa{\widehat {\mathfrak a}}

\newcommand\g{\mathfrak g}
\newcommand{\gog}{\mathfrak g}
\newcommand{\got}{\mathfrak t}

\newcommand{\gou}{\mathfrak u}
\newcommand{\gob}{\mathfrak b}
\newcommand{\gop}{\mathfrak p}
\newcommand{\goz}{\mathfrak z}
\newcommand{\goi}{\mathfrak i}
\newcommand{\goj}{\mathfrak j}

\newcommand{\gra}{\alpha}
\newcommand{\grb}{\beta} 
\newcommand{\grg}{\gamma}
\newcommand{\grd}{\delta}
\newcommand{\grl}{\lambda}
\newcommand{\grs}{\sigma}
\newcommand{\gre}{\varepsilon}

 \newcommand{\grD}{\Delta}  
 
\newcommand{\mk}  {\mathsf k}
\newcommand{\ra}         {\rightarrow}

\renewcommand{\geq}      {\geqslant}
\renewcommand{\leq}      {\leqslant}

\newcommand{\ab}       {\mathrm{ab}}
\newcommand{\ad}       {\mathrm{ad}}
\newcommand{\id}       {\mathrm{id}}
\newcommand{\mru}       {\mathrm u}
\newcommand{\wPsi}       {\widehat \Psi}
\newcommand{\wPhi}      {\widehat \Phi}

\DeclareMathOperator{\rk}{rk}

\begin{document}

\title[The Bruhat order on abelian ideals]{The Bruhat order on abelian ideals\\of Borel subalgebras}
\author[J.~Gandini]{Jacopo Gandini}
\author[A.~Maffei]{Andrea Maffei}
\author[P.~M\"oseneder Frajria]{Pierluigi M\"oseneder Frajria}
\author[P.~Papi]{Paolo Papi}
\date{\today}

\begin{abstract} 
Let $G$ be a quasi simple algebraic group over an algebraically closed field $\mk$ whose characteristic is not very bad for $G$, and let $B$ be a Borel subgroup of $G$ with Lie algebra $\gob$. Given a $B$-stable abelian subalgebra $\goa$ of the nilradical of $\b$, we parametrize the $B$-orbits in $\goa$ and we describe their closure relations.
\end{abstract}


\maketitle

\section{Introduction}

Let $G$ be a quasi simple algebraic group over an algebraically closed field $\mk$ whose characteristic is not very bad for $G$ (see Section \ref{ssez:preliminari} for the definition). Fix a Borel subgroup $B \subset G$  and a maximal torus $T \subset B$. Let $U$ be the unipotent radical of $B$. We denote the Lie algebras of $G$, $B$, $U$ respectively by $\gog, \gob$ and  $\u$. 

Let $\aa$ be a $B$-stable abelian subalgebra of $\u$, then $B$ acts on $\aa$ with finitely many orbits. When the characteristic of $\mk$ is good for $G$, this was noticed by G.~R\"ohrle \cite{Ro}, and then proved conceptually by the same author together with D.I.~Panyushev \cite{PR} (see also \cite{FR}), by showing a connection between the $B$-stable abelian subalgebras of $\u$ and the spherical nilpotent orbits in $\gog$. If the characteristic of $\mk$  is zero, the $B$-stable abelian subalgebras of $\u$ are exactly the abelian ideals of $\gob$. In this case, a combinatorial parametrization of the $B$-orbits was recently given by Panyushev \cite{Pa5}, by exhibiting canonical base points and establishing a parametrization in terms of the combinatorics of the root system of $G$.

The aim of this paper is to extend Panyushev's combinatorial parametrization of the $B$-orbits in $\aa$ to the case of positive characteristic, and to give a combinatorial description of the corresponding Bruhat order (that is, the partial order among the $B$-orbits given by the closure relations). This generalizes to arbitrary $B$-stable abelian subalgebras of $\u$ a conjecture of Panyushev \cite{Pa5} recently proved by the first two authors \cite{GM}, devoted to the case of the abelian ideals of $\gob$ which arise as the nilradicals of standard parabolic subalgebras of $\gog$.

The main ingredient which makes this generalization possible is the combinatorics of the affine Weyl group. More precisely, let $\Phi$ be the root system of $G$ associated to $T$ (regarded as the set of $T$-weights in $\g$), let $\grD \subset \Phi^+ \subset \Phi$ be the base and the set of positive roots determined by $B$, and set $\Phi^- = -\Phi^+$. Let $\widehat\Phi=\Phi\pm \mZ \grd$ be the corresponding affine 
root system ($\grd$ being the fundamental imaginary root), let 
$\widehat\Phi^+=(\Phi^++\mN \grd)\cup (\Phi^- +\mN_{>0}\delta)$
be a set of positive roots in $\widehat\Phi$ compatibly chosen with $\Phi^+$, and set $\widehat \Phi^- = - \widehat \Phi^+$.
 Let $W$ and $\widehat W$ be respectively the Weyl group and the  
affine Weyl group of $G$. Our choices determine sets of Coxeter generators, length functions,
and Bruhat orders $\leq$ for both $W$ and $\widehat W$.

Let $\widehat \gog =\gog[z,z^{-1}]\oplus \mk C \oplus \mk d$ be the affinization of $\gog$. For any real root $\gra\in \widehat\Phi$, we fix a root vector $e_\gra \in \widehat \gog$ and denote by $s_\gra \in \widehat W$ the associated reflection. More generally, if $S\subset \widehat \Phi$ is a finite set of pairwise orthogonal real roots (for short, an orthogonal set of roots) we define
$$
	\grs_S=\prod _{\gra\in S}s_\gra, \qquad e_S=\sum_{\a\in S} e_\a.$$
	
Let $\aa$ be a $B$-stable abelian subalgebra of $\u$. Then $\aa$ is a sum of root spaces of $\gob$, and we denote by $\Psi(\aa) \subset \Phi^+$ its set of $T$-weights. Panyushev \cite{Pa5} proved that the assignment $S \mapsto Be_S$ gives a bijection 
between the orthogonal subsets of $\Psi(\aa)$ and the $B$-orbits in $\aa$. This bijection carries over when the characteristic of $\mk$ is not very bad for $G$ (see Corollary \ref{cor:parametrizzazionedimensione}), and our main result is the following.

\begin{theorem*}[see Corollary \ref{cor:parametrizzazionedimensione} and Theorem \ref{teo:principale}]
Let $\aa$ be $B$-stable abelian subalgebra of $\u$ and let $S,S' \subset \Psi(\aa)$ be orthogonal subsets. Set $\widehat S = S-\grd$ and $\widehat S' = S'-\grd$, then
$$e_{S'} \in \overline{Be_S} \text{ if and only if } \grs_{\widehat S'}\leq \grs_{\widehat S}.$$
Moreover,
$$\dim Be_S=\frac{\ell(\grs_{\widehat S})+\card(S)}{2}.$$
\end{theorem*}

Notice that the previous theorem provides some sort of an affine formulation of the description conjectured by Panyushev in \cite{Pa5} and proved by the first two authors in \cite{GM}. Suppose indeed that $\goa = \gop^\mru$ is the nilradical of a parabolic subalgebra $\gop$, and let $W_P \subset W$ be the parabolic subgroup defined by $\gop$ and $w_P \in W_P$ its longest element. Given orthogonal subsets $S,S' \subset \Psi(\gop^\mru)$, Panyushev conjectured in this case that
$$
	e_{S'} \in \overline{Be_S} \text{ if and only if } \grs_{w_P(S')}\leq \grs_{w_P(S)},
$$
and that $\dim Be_S= \frac{1}{2} (\ell(\grs_{w_P(S)})+\card(S))$.

To explain the equivalence of the two formulations, recall that the nilradical of a standard parabolic subalgebra $\gop$ of $\gog$ is abelian if and only if $\gop$ is a maximal standard parabolic subalgebra, corresponding to a simple root $\gra_P \in \grD$ which occurs with multiplicity one in the highest root $\theta$ of $\Phi$.
If $\alpha_0 = \delta - \theta$ denotes the affine simple root, we can thus define an involution of the extended Dynkin diagram $\widehat \grD=\Delta\cup\{\a_0\}$ by setting
$\varphi(\a_P)=\a_0$ and $\varphi(\a)=-w_P(\a)$ for all $\gra\in \Delta\smallsetminus\{\gra_P\}$. The map $\varphi$ induces an involution 
of $\widehat W$ (still denoted by $\varphi$), which preserves the Bruhat order and the length function. Notice that $w_P(\alpha_P) = \theta$, and that $\alpha_P$ occurs with multiplicity one in every root in $\Psi(\gop^\mru)$. If moreover $\grb \in \Psi(\gop^\mru)$, then $\varphi(\beta) = \delta - w_P(\beta)$, yielding $\varphi(w_P(S)) = \grd - S$ and $\varphi(\sigma_{w_P(S)}) = \sigma_{\widehat S}$, which shows the equivalence of the two formulations.

The idea to use the affine root system  in the study of abelian ideals dates back to Peterson (see \cite{Ko}) and has been systematically used by many authors
(\cite{CP}, \cite{CP2}, \cite{Pa3}, \cite{Pa4}, \cite{Sommers}, \cite{Su}). Let us recall the key points in our setting. 
If $\aa$ is a $B$-stable abelian subalgebra of $\u$, then $\wgoa=z^{-1}\aa\subset \widehat\gog$ is a $B$-module isomorphic to $\aa$. 
Strictly speaking, all along the paper we will study the $B$-orbits in $\wgoa$: this is of course equivalent 
to the original problem. The set of $T$-weights of this $B$-module is the set of negative roots $\widehat\Psi(\aa) := \Psi(\aa)-\grd$,
and it has the fundamental property of being \textit{biclosed}: that is, both $\wPsi(\aa)$ and $\widehat\Phi^-\smallsetminus \wPsi(\aa)$ are closed under root addition. This property implies the existence of an element $w_\aa \in \widehat W$ such that 
$$
\wPsi(\aa)=\{\gra\in \wPhi^-\colon w_\aa (\gra)\in \wPhi^+\}.
$$

The assignment $\aa \mapsto w_\aa$ gives a bijective correspondence between the set of the $B$-stable abelian subalgebras of $\u$ and a remarkable class of elements of $\widehat W$. Set $\wPsi:=\Phi^+-\grd$. We say that an element $w\in \widehat W$ is \emph{minuscule} if
$$ \{\gra\in \wPhi^-\colon w(\gra)\in \wPhi^+\}\subset \wPsi,$$
and we denote by $\widehat W_{\ab}\subset \widehat W$ the set of minuscule elements.

Given $w \in \widehat W_\ab$, we define
$$ \widehat\Psi(w) = \{\gra\in \wPhi^-\colon w(\gra)\in \wPhi^+\}.$$
Correspondingly, we define two (isomorphic) $B$-stable subspaces of $\widehat \gog$ as follows
$$ \wgoa_w = \bigoplus_{\gra\in \wPsi(w)}\gog_\gra, \qquad \qquad  \goa_w=z\wgoa_w.$$
The space $\goa_w$ is easily seen to be a $B$-stable abelian subalgebra of $\u$, and the assignment $w \mapsto \goa_w$ gives a bijection 
between the minuscule elements of $\widehat W$ and the $B$-stable abelian subalgebras of $\u$. We note that usually one encodes minuscule elements by means of their positive inversions, namely by $-\wPsi(w)$. However this would lead to a $B^-$-stable subspace of $\widehat \gog$ (where $B^-$ denotes the opposite Borel subgroup of $B$ in $G$ with respect to $T$), and since we are presently interested in the orbit structure of the $B$-stable abelian subalgebras of $\u$, our choice seems preferable in this context.

The general strategy of the proof of our main theorem is similar to the one given in \cite{GM} in the case of the abelian nilradicals. 
As in that case, the most difficult part is to prove that the relations among the involutions imply the relations among 
the associated orbits closures. To deal with this problem, we consider a family of extensions of $B$ parametrised by $\widehat W_{\ab}$. More precisely, for $v \in \widehat W_\ab$, set $B_v = B\ltimes \widehat \aa_v$. These extensions have the following properties: if $v,w \in \widehat W_\ab$, then $B_v \subset B_w$ if and only if $v \leq w$, and $B_v$ acts on $\widehat \aa_w$ if and only if $v \leq w$ (see Proposition \ref{prop:weak-iff-strong} and formulas \eqref{eq:gruppo} -- \eqref{eq:azione}). This allows us to argue by decreasing induction on the length of $v \leq w$ in order to study of the $B_v$-orbit structure of $\widehat \aa_w$.

In particular, we analyse some basic relations among orbits whose dimensions differ by one. 
This analysis will be controlled by the descents of the involutions $\grs_S$, where $S$ varies among the orthogonal subsets of $\widehat \Psi(w)$.
In the case of the abelian nilradicals, this part of the analysis was performed by considering the action of the minimal parabolic subgroups of $G$ in the Hermitian symmetric variety associated to the abelian nilradical.
This is no more possible in the more general context of the $B$-stable abelian subalgebras of $\u$, however we are still able to perform our analysis by distinguishing the descents of the involutions $\grs_S$ into two families, 
the {\it finite descents} and the {\it affine descents} (see Definition \ref{def:discesefiniteaffini}), and adopting different strategies in the two cases.

The results and the techniques of the present paper are further extended in \cite{GMP} in order to study the $B$-orbits in $G\goa$ for a basic class of $B$-stable abelian subalgebras $\goa \subset \gou$.

The paper is organised as follows. Section \ref{ssez:preliminari}  is mainly devoted to preliminaries. In Subsection 
\ref{cp} we perform a detailed analysis, in positive characteristic,  of several notions which are equivalent in characteristic zero to being an abelian ideal of $\b$.

In Section \ref{sez:minuscoli} we recall and prove some properties of the minuscule elements. In Section \ref{sez:involuzioni} we prove
some properties of the involutions that we consider, and in particular we introduce the notions of finite and affine descents.
In Section \ref{sez:dimensione} we introduce the extensions $B_v$ and we give a classification of $B_v$-orbits which is slightly different from the one given in 
\cite{Pa5} and which allows to prove the formula for the dimension of an orbit. Finally in Section \ref{sez:principale} we describe the Bruhat order.

\noindent \textit{Acknowledgments.} We thank the anonymous referee for his/her useful remarks.

\section{Notation and preliminaries} \label{ssez:preliminari}

We expand here on the notation already set in the Introduction. Algebraic groups will be always denoted by capital letters, whereas their Lie algebras will be denoted by the corresponding fraktur letters. Recall that $G$ is a quasi simple algebraic group, and that $B \subset G$ is a fixed Borel subgroup with maximal torus $T \subset B$ and unipotent radical  $U$, so that $B=TU$.

A prime $p > 0$ is said to be \textit{very bad} for $G$ if $p=2$ and $G$ is of type $B_n,C_n, F_4$ or $G_2$, or if $p=3$ and $G$ is of type $G_2$, 
whereas it is called \textit{good} for $G$ if it does not divide any coefficient of the highest root of $\Phi$ 
(namely, $p\neq 2$ if $G$ is not of type $A_r$, $p \neq 3$ if $G$ is of type 
$E_r$, $F_4$ or $G_2$, and $p \neq 5$ if $G$ is of type $E_8$).

Let $\widehat \gog =\gog[z,z^{-1}]\oplus \mk C \oplus \mk d$. The algebra structure on $\widehat \gog$ is defined by letting   $C$ be a central element, $d$ act on $\widehat \gog$ by $z(\mathrm d/\mathrm dz)$ and setting 
$$
[z^nx,z^my]=z^{n+m}[x,y]+n\delta_{n,-m}\kappa(x,y)C,
$$
where $\kappa(\cdot,\cdot)$ is the Killing form on $\g$.
We point out that the the Killing form can be degenerate in positive characteristic, so $\widehat\gog$ is not, strictly speaking, an affine Kac-Moody algebra. However we will use only the combinatorics of the Weyl group and of the affine root system and, in Lemma 5.1, the description of the Bruhat order on the affine flag variety (which is only determined by the loop group, and not by its central extension). Thus we will use the terminology of affine Kac-Moody Lie algebra even though it is slightly improper.

We denote by $\widehat G$ the Kac-Moody group associated to $\widehat\g$, and by $\widehat T=T \times \mk^*_C \times \mk^*_d$
the maximal torus of $\widehat G$ containing $T$ whose Lie algebra contains $C$ and $d$. 
Let  $\widehat \Lambda$ be the character lattice of $\widehat T$. It carries a $\widehat W$-invariant nondegenerate bilinear form.



Then the sets of roots $\Phi\subset \widehat\Phi\subset \widehat \Lambda$ are respectively the set of $\widehat T$-weights in $\gog$ and in $\widehat \gog$. We regard the root lattice as a partially ordered set with the dominance order, defined by $\grl\leq \mu$ if $\mu-\grl$ is a sum of simple roots with nonnegative coefficients.

Given $\gra\in \widehat \Phi$, let $\gog_\gra \subset \widehat \gog$ be the corresponding root space and, if $\gra$ is real, let $u_\gra$ be the corresponding one parameter subgroup of $\widehat G$. In particular, if $\gra \in \Phi$ and $n \in \mZ$, we have $\gog_{\gra + n\grd} = z^n\gog_\gra$ and $u_{\gra + n\grd}(x)=u_\gra(z^n x)$. Given a real root $\gra$, recall that we can construct 
representatives of the reflections in the Weyl group as follows
$$ s_\gra= u_{-\gra}(-1)\,u_\gra(1)\,u_{-\gra}(-1). $$
We will denote by $\gra^\vee$ the coroot associated to $\gra$ and by $\langle .,. \rangle$ the bilinear pairing between roots and coroots, which we will also extend to their linear span.  In particular, we have $\langle \gra + n\grd, (\grb + m\grd)^\vee \rangle = \langle \gra,\grb^\vee \rangle$ for all $\gra,\grb\in \Phi$.

We finally make some remarks about our assumptions on the characteristic of $\mk$.
In positive characteristic, quasi simple groups with the same isogeny type could have different Lie algebras. 
However  $\gou$, the Lie algebra of the  unipotent radical of $B$,  is independent from the isogeny type of $G$, 
and the action of $B$ on $\gou$ always factors through the adjoint form of the group. 

As already recalled, our assumption on the characteristic implies that, if $\gra,\grb\in \Phi$ and $\gra+\grb\in \Phi$, then $[e_\gra,e_\grb]\neq 0$. 
Since a $B$-stable abelian subalgebra of $\u$ is a sum of root spaces, it follows that the classification of these ideals is equivalent to that of the corresponding sets of roots. On the other hand, a subset $\Psi_0 \subset \Phi^+$ is the set of roots of a $B$-stable abelian subalgebra of $\u$ if and only if it has the following two properties:
\begin{itemize}
	\item[i)] If $\gra \in \Psi_0$, $\grb \in \Phi^+$ and $\gra \leq \grb \in \Phi^+$, then $\grb \in \Psi_0$,
	\item[ii)] If $\gra, \grb \in \Psi_0$, then $\gra + \grb \not \in \Phi^+$.
\end{itemize}
A subset $\Psi_0 \subset \Phi^+$ which satisfies i) is usually called a \textit{combinatorial ideal} in $\Phi^+$; if it also satisfies ii) we will call it a \textit{combinatorial abelian ideal}. The previous discussion immediately implies that the classification of the $B$-stable abelian subalgebras of $\u$ over $\mk$ is the same as that in characteristic zero. In particular, under our assumption on the characteristic, all the properties of a $B$-stable abelian subalgebra of $\u$ which can be encoded in terms of the corresponding combinatorial abelian ideal do not depend on the characteristic of $\mk$.

Notice also that, since the characteristic of $\mk$ is not very bad for $G$,
 if $\gra,\grb \in \Phi$ are not proportional and if $m \in \mN$ is maximal such that
$\gra+m\grb$ is a root, then there exist nonzero constants $c_1,\dots,c_m \in \mk^*$ such that
$$ u_\grb(t) \cdot  e_\gra = e_\gra + c_1\, t\, e_{\gra + \beta}+\dots+c_m\, t^m\,e_{\gra+m\grb} $$ 
for all $t \in \mk$ (see for example the construction of Chevalley groups in \cite[Sections 1,2,3]{Steinberg}).

Note that in positive characteristic we do not always have an exponential map. However such a map can be defined in special cases. Suppose for instance that $X \subset \widehat \Phi$ is a finite set of real roots such that $\gra + \grb \not \in \widehat \Phi \cup \{0\}$ for all $\gra,\grb\in X$, then we can define an exponential map
\[\exp_X :\bigoplus_{\gra \in X} \mk\,e_\gra \rightarrow \widehat G,
\qquad \qquad \sum_{\gra\in X} t_\gra\,e_\gra \mapsto  \prod_{\gra}u_\gra(t_\gra),
\]
which is well defined since the elements $u_\gra$ commute by hypothesis. Notice that the previous assumption is satisfied when $X = \wPsi(w)$ for some $w\in \widehat W_{\ab}$, or more generally if $X$ is conjugated to such a subset by the action of $\widehat W$. Moreover, in this case we can normalise the elements $e_\gra$ so that the map 
$$\exp_{\wPsi(w)}:\wgoa_w \ra \widehat G$$
becomes equivariant for the adjoint action of $B$ on $\wgoa_w$ and on $\widehat G$: again this follows from
the construction of Chevalley groups given in \cite{Steinberg}. In what follows we will usually drop the subscript $X$, which will be 
clear from the context. 

\subsection{Abelian ideals of $\b$ in positive characteristic}\label{cp}

When the characteristic of $\mk$ is zero, the following notions all coincide:
\begin{enumerate}
\item[(a)] $B$-stable abelian subalgebras of $\u$;

\item[(b)]  $B$-stable abelian subalgebras of $\b$;
 
\item[(c)]  Abelian ideals of $\b$ contained in $\u$;
 
\item[(d)]  Abelian ideals of $\b$.
\end{enumerate}
In positive characteristic, the previous notions coincide if $G$ is adjoint, and there are the obvious implications $(a) \Rightarrow (b),(c),(d)$; $(b) \Rightarrow (d)$; $(c) \Rightarrow (d)$. However, in general, they are not equivalent. 
In the present paper we will use the notion (a).

In this subsection we explain the differences between the previous notions which arise when the characteristic of $\mk$ is positive. We always assume that the characteristic of 
$\mk$ (here denoted by $p$) is not very bad for $G$, and that $\Phi$ is irreducible.

We explain first the relations between the $B$-stable abelian subalgebras/ideals of $\u$ and the $B$-stable abelian subalgebras/ideals of $\b$. Let $Q \subset \Lambda$ be the root lattice of $\Phi$. 
If $p$ divides the cardinality of $\Lambda/Q$, then $\got$ contains a nontrivial subspace $\goz_\t$ which is central in $\g$, whose dimension is at most 2.
The group $B$ acts trivially on $\goz_\got$ and every $B$-stable abelian subalgebra $\goi \subset \gob$ splits as $\goi = \goi_\t \oplus \goi_\u$,
where $\goi_\t \subset \goz_\t$ and $\goi_\u \subset \gou$ is a $B$-stable abelian subalgebra of $\u$; vice versa any $\goi$ constructed in this way is a $B$-stable abelian subalgebra of $\b$. Hence the classification of the $B$-stable abelian subalgebras of $\b$ and the study of their $B$-orbits is easily reduced to the case of the $B$-stable abelian subalgebras of $\u$. The same stament holds for abelian ideals of $\b$, with the exception of the case $G=\mathrm{SL}(2)$, $p=2$. If we are not in this case (and $p$ is not very bad for $G$),
then every abelian ideal $\goi$ of $\b$ splits as
$\goi = \goi_\t \oplus \goi_\u$
where $\goi_\t \subset \goz_\t$ and $\goi_\u \subset \gou$ is an abelian ideal of $\u$, 
vice versa any $\goi$ constructed in this way is an abelian ideal of $\b$.


We now come to the relations between $B$-stable abelian subalgebras of $\u$ and abelian ideals of $\b$ contained in $\u$. 
The two concepts are equivalent unless we are in one of the following cases:
\begin{enumerate}
\item $G = \mathrm{SL}(3)$, $p=3$;
\item $G = \mathrm{SO}(2n)$, $p=2$, and $n\geq 3$;
\item $G = \mathrm{Spin}(2n)$, $p=2$, and $n\geq 3$.
\end{enumerate}
In cases (2) and (3), notice that for $n\geq 4$ we obtain groups of type $D_n$, whereas for $n=3$ we obtain groups of type $A_3$: $\mathrm{SL}(4)$ and its degree two quotient.

To prove the equivalence between the two concepts if we are not in one of the previous cases, 
one can prove that for all $\alpha, \beta \in \Phi$ with $\alpha \neq \pm\beta$ we have that 
$\mathrm d \alpha \neq \mathrm d \beta$.  This implies that every abelian ideal 
$\goi$ of $\b$ contained in $\u$ splits as $\goi = \bigoplus_{\alpha \in \Phi} 
\goi_\alpha$, where $\goi_\alpha = \goi \cap \g_\gra$. Thus, if we are not in 
one of the exceptional cases listed above, the classification of the abelian 
ideals of $\b$ contained in $\u$ is equivalent to that of the combinatorial 
abelian ideals of $\Phi^+$, namely the abelian ideals of $\b$ contained in 
$\u$ are precisely the $B$-stable abelian subalgebras of $\gou$.

Finally in the three exceptional cases listed above we provide examples of abelian ideals $\goi$ of $\b$ contained in $\u$ which are not contained in any 
$B$-stable abelian subalgebra of $\b$. 
If $G=\mathrm{SL}(3)$ we denote by $e_{ij}$ the elementary matrices, if instead $G=\mathrm{SO}(2n)$ or $\mathrm{Spin}(2n)$ 
we describe roots with the usual $\gre$-notation, and for $p=2$ we normalise the root vectors $e_\gra$ so that 
$[e_\gra,e_\grb]=e_{\gra+\grb}$ if $\gra,\grb,\grg$ are roots.

In case (1),  we take $\goi = \mk(e_{12}+e_{23})\oplus  \mk e_{13}$.

In cases  (2) and (3)  set 
\begin{gather*}
\alpha= \gre_2-\gre_n,\quad \beta=\gre_2+\gre_n, \quad  \eta=\gre_1-\gre_n,\quad \xi=\gre_1+\gre_n,\\
\Psi=\{\gre_1+\gre_i, \, \gre_2+\gre_i \; | \;  i\leq n-1\}.
\end{gather*}
Then  $\goj=\bigoplus_{\psi\in \Psi} \gog_\psi$ is an abelian ideal of 
$\gob$ (in any characteristic), and we take $\goi = \goj \oplus \mk(e_\alpha+e_\beta)\oplus \mk(e_\eta+e_\xi)$. When $n=3$ and $G=\mathrm{SL}(4)$, the latter example can be rewritten as $\goi = \mk(e_{12}+e_{34})\oplus  \mk(e_{13}+e_{24})\oplus\mk e_{14}$.

\section{Minuscule elements in the affine Weyl group}\label{sez:minuscoli}


In this section we will study some combinatorial properties of the minuscule elements and their sets of inversions $\widehat \Psi(w)$, which closely generalize similar properties of the minimal length coset representatives $W^P \subset W$ associated to a parabolic subgroup $P \subset G$ with abelian unipotent radical (see \cite[Section 2]{GM}).

Let $w\in \Wab$. Since $\wPsi(w)+\grd$ is a combinatorial ideal in $\Phi^+$, the set of inversions $\wPsi(w)$ has the following property, that we will use very often:
\begin{equation} 	\label{eq:facile} 
\text{if } \gra \in \wPsi(w) \text{ and } \grb\in \wPsi \text{ satisfy } \gra\leq\grb, \text{ then } \grb\in \wPsi(w).
\end{equation}

Indeed in this case $\grb-\gra$ is a sum of simple roots in $\Delta$, which by definition are transformed into positive roots by $w$, thus $w(\grb)\geq w(\gra)>0$.

In particular, it follows that two roots $\alpha, \beta \in \widehat \Psi(w)$ are orthogonal if and only if they are strongly orthogonal (that is, $\gra \pm \grb \not \in \widehat \Phi \cup \{0\}$). Indeed, since $\widehat\Psi(w)$ is a subset of $\widehat \Psi$ closed under root addiction, the sum of two roots in $\widehat\Psi(w)$ is never a root. On the other hand, if $\gra , \grb \in \widehat \Phi$ are orthogonal and $\gra - \grb \in \widehat \Phi$, then $\gra + \grb = s_\grb(\gra - \grb)$ is a root as well.

The following useful property was noticed by Panyushev \cite[Lemma 1.2]{Pa5}. We provide a proof in a slightly different language.

\begin{lemma}\label{lemma:adding-roots}
Let $w\in \Wab$. Let $\gra, \grb \in \widehat\Psi(w)$ be orthogonal roots and let $\grg \in \Phi$ be such that $\gra + \grg \in \widehat\Psi(w)$, then $\grb + \grg \not \in \widehat\Psi(w)$. In particular, if $S\subset\widehat\Psi(w)$ is orthogonal and $\grg\in \Phi^+$, then there is at most one root $\a \in S$ such that $\a+\gamma\in \widehat \Phi$. 
\end{lemma}

\begin{proof}
Assume that $\gra + \grg$, $\grb + \grg$ are both in $\widehat\Psi(w)$.
Since $w\in\Wab$,  we have that 
$$\langle \gra, \grb^\vee \rangle = 0 \iff \gra-\grb\notin \Phi \iff (\gra+\grg)-(\grb+\grg)\notin \Phi \iff \langle \gra+\grg, (\grb+\grg)^\vee \rangle = 0.$$

By the orthogonality of $\gra$ and $\grb$, the last equality implies that either  $\langle \gra, \grg^\vee \rangle <0$ or $\langle \grb, \grg^\vee \rangle <0$. Suppose that $\langle \gra, \grg^\vee \rangle <0$: then $\langle \grb+ \grg, \gra^\vee \rangle  < 0$, thus $\gra + \grb + \grg \in \widehat\Psi(w)$, contradicting $w\in \Wab$. Similarly, it cannot be $\langle \grb, \grg^\vee \rangle < 0$. 

Suppose now that $S \subset \widehat \Psi(w)$ is orthogonal, and let $\gra \in S$ and $\grg\in \Phi^+, \a\in S$ be such that $\a+\gamma \in \widehat \Phi$. Then $\a+\gamma\in \widehat\Psi(w)$, thus $\beta+\gamma \not \in \widehat \Phi$ for all $\beta\in S \smallsetminus \{\gra\}$.
\end{proof}

We now recall some facts about the correspondence between minuscule elements and $B$-stable abelian subalgebras of $\u$. By \cite{CP}, we have the following important characterization:
\begin{equation}	\label{eq:alcova}
	\Wab = \{w\in\Wa\mid w^{-1}\mathcal A\subset 2\mathcal A\},
\end{equation} 
where $\mathcal A$ denotes the fundamental alcove for $\Wa$ in the vector space generated by $\Phi^\vee$.

The following theorem, due to Panyushev, shows how to describe the normalizer of a $B$-stable abelian subalgebra of $\u$ (which is by definition a standard parabolic subgroup) in terms of the corresponding minuscule element.

\begin{theorem}[{\cite[Theorem 2.8]{Pa4}}] 	\label{teo:panyushev-normalizzatore}
Let $w \in \Wab$ and $\gra \in \grD$, then $\gra$ is a simple root of $N_G(\aa_w)$ if and only if $w(\gra) \in \widehat \grD$. 
\end{theorem}

A remarkable property of $\Wab$ is that the restriction of the Bruhat order coincide with that of the right weak order.

\begin{proposition} \label{prop:weak-iff-strong}
Let $w_1,w_2 \in \Wab$, then $w_1 \leq w_2$ if and only if $\widehat\Psi(w_1) \subset \widehat\Psi(w_2)$.
In particular,  $w_1\leq w_2$ if and only if $\aa_{w_1}\subset \aa_{w_2}$.
\end{proposition}

\begin{proof}
If $\widehat\Psi(w_1) \subset \widehat\Psi(w_2)$, then $w_1$ is less or equal than $w_2$ in the right weak Bruhat order on $\widehat W$, thus $w_1\leq  w_2$.

Assume vice versa that $w_1 <  w_2$. Recall that every minuscule element is a minimal length representative for its coset in $\Wa/W$. By the chain property (see \cite[Theorem 2.5.5.]{BB}), there exists a chain in $\widehat W$ 
$$w_1=v_1<v_2\ldots<v_k=w_2$$
such that, for every $i  >0$, the element $v_i$ has minimal length in its coset in $\Wa/W$ and $\ell(v_{i+1})=\ell(v_i)+1$.

We now show that $v_i \in \Wab$ for all $i$. Since $v_i < v_i s_\gra$ for all $\gra \in \Phi$, the alcoves $\calA$ and $v_i^{-1} \calA$ are never separated by an hyperplane of the shape $\langle \gra, x \rangle = 0$ with $\gra \in \Phi$ (see \cite[Theorem 4.5]{Hu}): thus $\langle \gra, x \rangle > 0$ for all $x \in v_i^{-1}\calA$ and for all $\gra \in \Phi^+$. Therefore, to show that $v_i \in \Wab$, by \eqref{eq:alcova} it only remains to prove that $\langle \theta, x \rangle < 2$ for all $x \in v_i^{-1} \calA$, where $\theta \in \Phi^+$ denotes the highest root. Proceeding inductively, we can assume that $v_{i+1} \in \Wab$.

Since $\ell(v_{i+1})=\ell(v_i) + 1$, we have   $v_i = v_{i+1} s_\grb$ for some real negative root $\grb \in \widehat \Phi$ such that $v_{i+1}(\grb) > 0$. Since $v_{i+1} \in \widehat W_\ab$,  it follows that $\grb \in \widehat \Psi(v_{i+1})$, thus $\grb = \gra - \delta$ for some $\alpha \in \Phi^+$.
In the vector space generated by the coroots, write $s_\grb(x) = x-( \langle \a, x \rangle -1) \, \a^\vee$ and let $H_\grb$ be the affine hyperplane defined by $s_\grb$. Since $v_{i+1}  s_\grb < v_{i+1}$, the hyperplane $H_\grb$ separates the alcoves $\mathcal A$ and $v_{i+1}^{-1}\mathcal A$ (see \cite[Theorem 4.5]{Hu}). On the other hand $H_\grb$ is defined by the equation $\langle \gra, x \rangle = 1$, thus $\langle \gra, x \rangle > 1$ for all $x \in v_{i+1}^{-1}\calA$. Since by assumption $v_{i+1} \in \widehat W_\ab$, by \eqref{eq:alcova} we have $v_{i+1}^{-1}\calA \subset 2 \calA$, hence $\langle \theta, x \rangle < 2$  for all $x \in v_{i+1}^{-1}\calA$. Let now $x \in v_i^{-1}\calA$ and set $y = s_\grb(x)$: then $y \in v_{i+1}^{-1}\calA$, thus
$$
	\langle \theta,x \rangle  =  \langle \theta, y \rangle - (\langle \gra, y \rangle-1) \langle \theta, \gra^\vee \rangle < \langle \theta, y \rangle < 2,
$$
which shows that $v_i \in \Wab$.
 
It remains to prove that $\widehat\Psi(w_1)\subset \widehat\Psi(w_2)$. Thanks to the argument above, we can assume that $\ell(w_2)=\ell(w_1)+1$. 
Thus they exist $v \in \widehat W$ and simple roots $\a_1, \ldots, \gra_{i+1} \in \widehat \grD$ such that $w_1 = s_{\a_1}\cdots s_{\gra_i} v$ and $w_2 =  s_{\a_1}\cdots s_{\gra_{i+1}} v$, with 
$\ell(w_1) = \ell(v)+i$ and $\ell(w_2)=\ell(v)+i+1$. Since $w_1, w_2$ are minuscule, the roots $v^{-1}(\alpha_i)$, $v^{-1}(\alpha_{i+1})$ and $v^{-1} s_{\alpha_{i+1}}(\alpha_i)$ all belong to $- \widehat \Psi$, thus there exist $\beta_1, \beta_2, \beta_3\in \Dp$ such that
$$v^{-1}(\a_i) = \d-\beta_1, \qquad v^{-1}(\a_{i+1}) = \d-\beta_2, \qquad v^{-1}s_{\alpha_{i+1}}(\a_i) = \d-\beta_3.$$
On the other hand, the equality
$$\d-\beta_3 = v^{-1}(\a_i - \langle \a_i, \a_{i+1}^\vee \rangle  \a_{i+1})  =
(\d-\beta_1) - \langle \a_i, \a_{i+1}^\vee \rangle  (\d-\beta_2),$$
forces $\langle \a_{i+1}, \a_i^\vee \rangle  = 0$: this shows that $s_{\gra_i}$ and $s_{\gra_{i+1}}$ commute, and repeating the argument we see that $w_2 = s_{i+1} w_1$.
\end{proof}

\begin{remark}
Notice that the previous proof also shows that $\Wab$ is downward closed in the set of the minimal length coset representatives for $\widehat W/W$ with the Bruhat order: that is, if $w_1 \in \widehat W$ has minimal length in its coset $w_1 W$ and if $w_1\leq w_2$ for some $w_2 \in \widehat W_\ab$, then $w_1 \in \widehat W_\ab$ as well (see \cite[Lemma 2.4]{Ku}).

Notice also that the statement of the previous proposition needs not to be true if we consider more generally minimal length coset representatives for $\widehat W/W$. Take for instance $\Phi$ of type $A_2$, $w_1 = s_1 s_0$ and $w_2 = s_1s_2s_0$.
\end{remark}

The following proposition describes the left descents of a minuscule element.

\begin{proposition}\label{prop:inversioni minimali e massimali}
Let $w\in \Wab$.
\begin{enumerate}[\noindent i)]
\item Let $\gra\in \wPhi^+$ be such that $s_\gra w<w$ and assume that $s_\gra w\in \widehat W_{\ab}$ and $\ell(s_\gra w)=\ell(w)-1$.  
Then $\gra \in \widehat \grD$. 

\item Let $\gra \in \widehat \Delta$ be such that $s_\gra w < w$, then $s_\gra w \in \Wab$. If moreover $\grb=w^{-1}(\gra)$, then $\grb$ is minimal in $\wPsi(w)$ and maximal in $\wPsi \smallsetminus \wPsi(s_\gra w)$, and we have the equality $\wPsi(w)=\wPsi(s_\gra w)\cup\{\grb\}$.

\item If $\grb$ is a minimal element in $\wPsi(w)$ and $\gra = w(\grb)$, then $\gra \in \widehat \grD$ and $s_\gra w<w$.
\end{enumerate}
\end{proposition}

\begin{proof}
i) By Proposition \ref{prop:weak-iff-strong} we have $\wPsi(s_\gra w) \subset  \wPsi(w)$, thus $\gra$ 
is simple by \cite[Proposition 3.1.3]{BB}.

ii) Since $\gra \in \widehat \grD$, it follows that $s_\gra w$ is smaller than $w$ with respect to the left weak Bruhat order. Thus $s_\gra w\in \widehat W_{\ab}$ because $w \in \widehat W_{\ab}$, and $\wPsi(w)=\wPsi(s_\gra w) \cup \{\grb\}$. 
The other claims follow from \eqref{eq:facile}.

iii) If $\grb$ is minimal in $\wPsi(w)$, then $\wPsi(w)\smallsetminus \{\grb\}$ is also biclosed, thus it must be equal to 
$\wPsi(w')$ for some $w'\in\Wab$. By Proposition \ref{prop:weak-iff-strong} we have $w'<w$. On the other hand $\ell(w') = \ell(w)-1$, thus it must be $w'=s_\gra w$ for some $\gra \in \widehat \Phi^+$ and we conclude thanks to i) and ii).
\end{proof}

\begin{remark}
Let $w \in \Wab$ and let $\grb \in \widehat \Psi \smallsetminus \widehat\Psi(w)$ be a maximal element, then $w(\grb)$ is not necessarily the opposite of a simple root. 
Take for instance $\Phi = A_3$, $w=s_1s_3s_0$ and  $\beta= \a_1 -\d$: then $\grb$ is maximal in $\widehat\Psi \smallsetminus \widehat\Psi(w) = 
\{\a_1 -\d, \a_2 -\d, \a_3-\d\}$, but $w(\beta)= - \a_1 -\a_2$. 

On the other hand, if there exists $v\in\Wab$ such that $w < v$ and $\grb$ is maximal in $\widehat\Psi(v)\smallsetminus \widehat\Psi(w)$, then $\a=-w(\beta)$ is a simple root and $s_\a w>w$. Indeed, it is enough to prove that $L={\widehat\Psi(w)\cup\{\grb\}}$ is biclosed. 
Since $L\subset {\wPsi(v)}$, It is clear that $L$ is closed under root addition. To check that $\wPhi^- \smallsetminus L$ is also closed under root addition, 
it is enough to check that, if $\grb = \grg+ \eta$ for some negative roots $\gamma, \eta$, then either $\gamma$ or $\eta$ belongs to $\wPsi(w)$.
We may assume that $\gamma= \grg_0-\d$ and $\eta\in\Phi^-$, thus $\grg\in \wPsi(v)$ since $\wPsi(v)$ is biclosed and $\eta \not \in \wPsi(v)$. Since $\grb=\grg+\eta$ we get then $\grb<\grg$, hence $\grg\in \wPsi(w)$  by the
maximality of $\grb$ in $\widehat\Psi(v)\smallsetminus \widehat\Psi(w)$.
\end{remark}

\section{Affine involutions associated to orthogonal sets of roots\\ and admissible pairs.}\label{sez:involuzioni}

Recall that for all finite set of pairwise orthogonal real roots $S \subset \widehat \Phi$ we defined an involution in $\widehat W$ by setting
$$ \grs_S = \prod_{\gra \in S} s_\a. $$

\begin{proposition}	\label{prop:radici-attive-reali}
Let $w\in\Wab$ and let $S \subset \widehat\Psi(w)$ be an orthogonal subset. 
Let $\gra \in \widehat \Phi$ be such that $\grs_S(\gra) = - \gra$, 
then $\a=\half(\pm\grb\pm\grb')$ with $\grb,\grb'\in S$. In particular:
\begin{enumerate}[\indent i)]
	\item If $\a\in \widehat\Psi$, then $\a=\half(\grb+\grb')$ with $\grb,\grb'\in S$. 
	\item If $\a\in \Phi$, then $\a=\half(\grb-\grb')$ with $\grb,\grb'\in S$.
	\end{enumerate}
\end{proposition}

\begin{proof}
Since $\grs_S(\gra) = -\gra$, we have
\begin{equation}	\label{eq:2alpha}
2\gra = \sum_{\grb\in S}\langle \gra,\grb^\vee\rangle\grb.
\end{equation}
Thus the claim follows if we show that $\sum_{\beta\in S}|\langle \gra,\grb^\vee\rangle|\le 2$.

Since the statement is purely combinatorial, we can assume that the characteristic of $\mk$ is zero. Consider the $\mathfrak{sl}_2$-triple $\{e,h,f\}$ in $\gog$ defined by the elements
$$
	e=\sum_{\beta\in S} e_{\d+\beta}, \qquad h=\sum_{\beta\in S} (\d+\beta)^\vee, \qquad f=\sum_{\beta\in S} e_{-\beta-\d}.
$$

Since $e \in \aa_w$, by a result of Panyushev and R\"ohrle \cite[Proposition 2.2]{PR} (see also the proof of \cite[Theorem  2.3]{PR}) it holds $\ad(e)^4=0$. This implies that $\langle \eta, h \rangle \leq 3$ for all $\eta\in\Phi$: indeed, as easily follows from the theory of $\mathfrak{sl}_2$-representations, the greatest eigenvalue of $\ad(h)$ on $\gog$ equals the greatest $n \in \mN$ such that $\ad(e)^n \neq 0$.

Denote $S'=\{\grb\in S\mid \langle \gra,\grb^\vee\rangle< 0\}$, then taking $\eta = \grs_{S'}(\gra)$ we get 
$$
	\sum_{\beta\in S}|\langle \gra,\grb^\vee\rangle| = \sum_{\grb\in S}\langle \grs_{S'}(\gra),\grb^\vee\rangle\le 3.
$$
On the other hand, by \eqref{eq:2alpha} we have
$$
	\grs_{S'}(\alpha) = \gra - \sum_{\grb \in S'} \langle \gra,\grb^\vee\rangle\grb =  \frac{1}{2}\sum_{\grb\in S}|\langle \gra,\grb^\vee\rangle| \grb.
$$
Therefore $\sum_{\beta\in S}|\langle \gra,\grb^\vee\rangle| = - 2 \langle \grs_{S'}(\gra), d \rangle$ is an even number, and the first claim is proved.

Suppose now that $\a=\half(\pm\grb\pm\grb')$ with $\grb, \grb' \in S$, and assume that $\gra\in\widehat\Psi$: then $-1= \langle \gra, d \rangle =\half(\pm1\pm1)$, thus $\a=\half(\grb+\grb')$. 
Similarly, if $\gra\in\Phi$, then $0=\langle \gra, d \rangle$, thus $\a=\pm\half(\grb-\grb')$.
\end{proof}

\begin{remark}\label{ce}
Proposition 
\ref{prop:radici-attive-reali} is in general false if we do not assume  
$S\subset \widehat\Psi(w)$ for some $w\in\Wab$. Take for example 
$\Phi = D_4$, 
$S=\{\varepsilon_1+\varepsilon_3-\d,\varepsilon_1-\varepsilon_3-\d,\varepsilon_2+\varepsilon_4-\grd,\varepsilon_2-\varepsilon_4-\grd \}$ and 
$\a=\varepsilon_1+\varepsilon_2-2\grd=\frac{1}{2}\sum_{\beta\in S}\beta$. 
\end{remark}

\begin{proposition}	\label{prop:iniettivita-involuzioni}
Let $w\in\Wab$ and let $S \subset \widehat\Psi(w)$ be an orthogonal subset.  If $S'\subset \widehat\Psi$ is an orthogonal subset such that $\grs_{S'} = \grs_S$, then $S' = S$.
\end{proposition}

\begin{proof}If $\a\in S'$, then 
$\grs_{S'}(\gra)=\grs_S(\gra)=-\gra$, hence $2\gra = \sum_{\grb\in S}\langle \gra,\grb^\vee\rangle\grb$ and
$\gra\in \Span_\mQ S$. Therefore $S'\subset \Span_\mQ S$, and switching the role of $S$ and 
$S'$ we get $ \Span_\mQ S'= \Span_\mQ S$. Assume now that $S'\ne S$ and set $S=\{\grb_1,\ldots,\grb_k\}$. 
Let $\gra\in S' \smallsetminus S$, then by Proposition \ref{prop:radici-attive-reali} we can write $\gra=\tfrac 12 (\grb+\grb')$ with $\grb,\grb'$ in $S$. Without loss of generality we can assume that $\gra=\tfrac 12 (\grb_1+\grb_2)$. 
Let $\gra'\in S'\smallsetminus \{\gra\}$, and write $\gra'=\sum a_i\grb_i$. Since $\gra$ and $\gra'$ are orthogonal, it follows that $a_1\Vert \grb_1\Vert^2 +a_2\Vert\grb_2\Vert^2=0$, 
hence $a_1=-\frac{ \Vert\grb_2\Vert^2}{ \Vert \grb_1\Vert^2} a_2$. On the other hand, by Proposition \ref{prop:radici-attive-reali} we can write $\gra'=\tfrac 12(\grb_i+\grb_j)$ for some $\grb_i, \grb_j \in S$, and since $\gra' \neq \gra$ we get $a_1=a_2=0$. It follows that $S'\smallsetminus \{\gra\}\subset \Span_\mQ(S\smallsetminus\{\grb_1,\grb_2\})$, against the fact that $S$ and 
$S'$ span the same space.
\end{proof}

\begin{remark}\label{nls}
Notice that the claim of Proposition \ref{prop:iniettivita-involuzioni} 
is false if we consider any pair of  strongly orthogonal subsets 
$S,S' \subset \widehat \Psi$. Take for example $\Phi = D_4$, $S$ as in Remark \ref{ce} and $S'=\{\varepsilon_1+\varepsilon_4-\d,\varepsilon_1-\varepsilon_4-\d, \varepsilon_2+\varepsilon_3-\grd, \varepsilon_2-\varepsilon_3-\grd \}$, then $\grs_{S'} = \grs_S$.
\end{remark}

The following definitions are adapted from Richardson and Springer \cite{RS1}.

\begin{definition}	\label{def:discese}
Let $\grs \in \widehat W$ be an involution and $\gra \in \widehat \Delta$.

We say that $\gra$ is a \textit{descent} for $\grs$ if  $\grs(\gra) < 0$. 
If moreover $\grs(\gra) = - \gra$, then we say that $\gra$ is a \textit{real descent}, otherwise we say that $\gra$ is a \textit{complex descent}.

We also define
$$
s_\gra \circ \grs=
\begin{cases}
 s_\gra \grs &\text{if } s_\gra \grs = \grs s_\gra \\
 s_\gra \grs s_\gra &\text{if } s_\gra \grs \neq \grs s_\gra
\end{cases}
$$

Finally we define the \emph{length} of $\grs$ as 
$$
L(\grs)=\frac{\ell(\grs)+\rk (\id-\grs)}{2}.
$$
\end{definition}

Notice that the length of an involution is a natural number. Indeed $\rk (\id-\grs)$ equals the multiplicity of $-1$ as an eigenvalue for $\grs$,
thus $\det(\grs) = (-1)^{\ell(\grs)} = (-1)^{\rk (\id-\grs)}$, and the parity of $\ell(\grs)$ equals that of $\rk (\id-\grs)$.

If $\grs$ is an involution, notice that $s_\gra \circ \sigma$ is an involution as well, and that $s_\gra \circ \grs = \tau$ if and only if $s_\gra \circ \tau = \grs$. If moreover $\grs=\grs_S$ for some finite orthogonal set of real roots $S$, then $L(\grs_S)=\tfrac 12 (\ell(\grs_S)+\card S)$.

In the next two lemmas we recall some simple results of Richardson and Springer from \cite{RS1,RS2}, see also \cite{GM} for more direct proofs. First we characterize the descents of an involution.

\begin{lemma} \label{lemma:discese-CR}
Let $\grs \in \widehat W$ be an involution and let $\gra\in \widehat\Delta$, then the following statements are equivalent: 
\begin{enumerate}[\indent i)]
        \item $\gra$ is a descent for $\grs$; 
        \item $s_\gra\grs<\grs$;
        \item $\grs\,s_\gra<\grs$;
 	\item $s_\gra\circ \grs <\grs$; 
        \item $L(s_\gra\circ \grs)=L(\grs)-1$.
\end{enumerate}
If moreover $\alpha$ is a descent for $\grs$, then it is real if and only if $s_\gra\,\grs=\grs \, s_\gra$, and it is complex if and only if
$s_\gra\grs\,s_\gra<s_\gra \grs$ and $s_\gra\grs\,s_\gra<\grs \,s_\gra $. 
\end{lemma}

\begin{proof}
The only thing that it is not proved in \cite[Section 3]{GM} is that, if $\gra$ is a descent for $\grs$, then it is real
if and only if $s_\gra$ and $\grs$ commute. On the other hand, if they commute, then $s_\gra=\grs s_\gra \grs=s_{\grs(\gra)}$, and being $\grs(\gra)<0$ we deduce $\grs(\gra)=-\gra$.
\end{proof}

\begin{lemma} [{see \cite[Lemma 3.2]{GM}}] \label{lemma:par}
Let $\gra\in \widehat \Delta$ and let $\grs, \tau \in \widehat W$ be involutions such that $\grs < \tau$. Then the following hold:
\begin{enumerate}[\indent i)]
 \item  if $s_\gra \circ \grs > \grs$ and $s_\gra \circ \tau > \tau$, then $s_\gra \circ \grs < s_\gra \circ \tau$;
 \item  if $s_\gra \circ \grs < \grs$ and $s_\gra \circ \tau < \tau$, then $s_\gra \circ \grs < s_\gra \circ \tau$; 
 \item  if $s_\gra \circ \grs > \grs$ and $s_\gra \circ \tau < \tau$, then $s_\gra \circ \grs \leq \tau$ and $\grs \leq s_\gra \circ \tau$.
\end{enumerate}
\end{lemma}

We now study the descents of the involutions of the form $\grs_S$, with $S$ an orthogonal set of roots contained in $\wPsi(w)$ for some $w\in \Wab$.

\begin{proposition}\label{prop:sbetaammissible}
Let $w \in \Wab$, let $S\subset \widehat\Psi(w)$ be an orthogonal subset and let $\gra\in\grD$ be a descent for $\grs_S$. 
Then $s_\gra(S)\subset \widehat\Psi(w)$, and $s_\gra(S) = S$ if and only if $\gra$ is a real descent.
\end{proposition}

\begin{proof}
We first prove the second claim, which is easier. 
Suppose that $\grs_S(\gra) = -\gra$, then by Proposition \ref{prop:radici-attive-reali} we have $\gra = \half(\grb_1 - \grb_2)$ 
for some $\grb_1, \grb_2 \in S$, 
thus $\grb_1$ and $\grb_2$ are switched by $s_\gra$, and $s_\gra(\grb) = \grb$ for all $\grb \in S$ different from $\grb_1$, $\grb_2$.

Suppose conversely that $s_\gra(S) = S$ for some $\gra \in \grD$ with $\grs_S(\gra) < 0$. Then $s_\gra$ and $\grs_S$ commute, hence
$\gra$ is a real descent by Lemma \ref{lemma:discese-CR}.

Assume now that $\grs_S(\gra) \neq - \gra$. Since $\grs_S(\gra) = \gra - \sum_{\grb \in S}\ \langle \gra, \grb^\vee \rangle \grb$ is negative and since the elements in $S$ are also negative, 
there exists $\grb_0 \in S$ such that $\langle \gra, \grb_0^\vee \rangle < 0$. 
Thus $\grb_0 + \gra \in \widehat\Psi(w)$, and $\grb_0$ is unique thanks to Lemma \ref{lemma:adding-roots}. Notice that 
$s_\gra(\grb_0)>\grb_0$ and that $\langle s_\gra(\grb_0), d \rangle = -1$, thus $s_\gra(\grb_0) \in \widehat \Psi(w)$ as claimed. 
Let now $\grb_1 \in S \smallsetminus \{\grb_0\}$. If $\grb_1$ is orthogonal to $\gra$, then $s_\gra(\grb_1)=\grb_1$ and the claim is trivial. 
Otherwise, it must be $\langle \gra, \grb_1^\vee \rangle > 0$. Set $n_0=-\langle \gra, \grb_0^\vee \rangle $, $n_1=\langle \gra, \grb_1^\vee \rangle$ and $\grg = s_{\grb_0}s_{\grb_1}(\gra)$. Then we have
$$
	\grg =s_{\grb_0}s_{\grb_1}(\gra)=\gra+n_0\grb_0-n_1\grb_1\leq \grs_S(\gra)<0.
$$ 
In particular we can assume that $S=\{\grb_0,\grb_1\}$. Notice also that $\langle \grg, d \rangle = n_1-n_0$, thus $n_0\geq n_1$. 

If $\Phi$ is of type $G_2$, then there are no combinatorial abelian ideals in $\Phi^+$ containing two orthogonal roots, so we are not in this case.

If $n_0=n_1$ and $\langle \grb_1, \gra^\vee\rangle=1$, then $\langle \grg, d \rangle = 0$. Therefore $\gra+n_0\grb_0 -n_0\grb_1= \gamma \in \Phi^-$, and since $\gra$ is simple 
we deduce $s_\gra(\grb_1) = \grb_1-\gra > \grb_0$. Thus we get $s_\gra(\grb_1) \in \widehat\Psi(w)$ by \eqref{eq:facile}.

If $n_0>n_1$, then $n_0=2$, $n_1=1$. In this case $\gra$ and $\grb_1$ are long roots, whereas $\grb_0$ is a short root, and $s_\gra(\grb_1)=\grb_1-\gra$.
Notice that $\gra-\grb_1$ and $\grg=\gra-\grb_1+2\grb_0$ are both roots, hence $\gra-\grb_1+\grb_0$ is either a root or zero. Denote $\gamma' =\gra-\grb_1+\grb_0$. If $\grg'\leq 0$ then we can conclude as in the previous case. Otherwise $\grg'\in \Phi^+$, and we get a contradiction because $\grg'+\grb_1=\gra+\grb_0
=s_{\gra}(\grb_0)$ and $\grg'+\grb_0=\grg$ are both roots, which is impossible by Lemma \ref{lemma:adding-roots}.

If $n_0=n_1$ and $\langle \grb_1, \gra^\vee\rangle=2$, then $n_0=n_1=1$. Moreover, $\grb_0$ and $\grb_1$ are long roots, whereas $\gra$ and $\grg$ are short roots, and $\langle \grb_0, \gra^\vee\rangle=-2$. 
Thus $\langle \grg,\gra^\vee\rangle = -2$ and $\langle \grg, d \rangle = 0$, and it follows $\gra + \grb_0 - \grb_1 = \grg \in \Phi^-$. Since $\gra$ is simple and $\grg + 2\gra \in \Phi$, we have $\grg + \gra \in \Phi^- \cup \{0\}$. If $\grg = - \gra$, then $s_\gra(\grb_1) = \grb_0 \in \widehat \Psi(w)$. Otherwise $2 \gra + \grb_0 - \grb_1 = \gra + \grg \in \Phi^-$, and it follows $s_\gra(\grb_1) = \grb_1-2\gra > \grb_0$, thus $s_\gra(\grb_1) \in \widehat\Psi(w)$ by \eqref{eq:facile}.
\end{proof}

\subsection{Admissible pairs and their associated involutions}

\begin{definition}
Let $v \in \Wab$ and $S \subset \widehat \Psi$ an orthogonal subset. Given $w \in \Wab$, the pair $(v,S)$ is called $w$-\textit{admissible} 
if $v\leq w$ and $S \subset \widehat\Psi(w) \smallsetminus \widehat\Psi(v)$. The pair $(v,S)$ is called \textit{admissible} if it is $w$-admissible for some 
$w \in \Wab$.
\end{definition}


Given an admissible pair $(v,S)$ we will sometime refer to $\grs_{v(S)}$ as the 
\textit{involution associated to} $(v,S)$, and will denote it also by $\grs(v,S)$.

\begin{definition}
Given $\gra \in \widehat \grD$, we say that $\gra$ is a \textit{descent} for $(v,S)$ if $\grs_{v(S)}(\gra) < 0$.
\end{definition}

We will make use of the terminology introduced in Definition \ref{def:discese} also in this context: we will say that a descent 
$\gra$ for an admissible pair $(v,S)$ is \textit{real} (resp. \textit{complex}) if it is so as a descent for $\grs_{v(S)}$, and we will write $L(v,S)$ for $L(\grs(v,S))$.


The goal of this subsection is to study the descents associated to admissible pairs. The geometrical meaning of the admissible pairs and of their descents, which will play a major role in our inductions, will become clear in the next sections, 
when we will give an interpretation of the admissible pairs in terms of $B$-orbits in $B$-stable abelian subalgebras of $\u$.

The next proposition shows that a descent of an admissible pair is of two possible types. 

\begin{proposition}	\label{prop:discese-AF}
Let $w \in \Wab$, let $(v,S)$ be a $w$-admissible pair and let $\gra \in \widehat \grD$ be a descent for $(v,S)$. 
Set $\grb = v^{-1}(\gra)$. Then $\grb>0$, and either $\grb \in -\widehat\Psi(w)$ or $\grb \in \grD$. If moreover $\grb \in \grD$, then $\grs_S(\grb)<0$.
\end{proposition}

\begin{proof}
By the orthogonality of $S$ we have
\begin{equation}	\label{eq:sigma-vSa}
\grs_{v(S)}(\gra) =  \gra - \sum_{\grg \in S} \langle \grb , \grg^\vee \rangle v(\grg).
\end{equation}
Since by assumption $\grs_{v(S)}(\gra)$ and $v(\grg)$ are both negative when $\grg\in S$,
there is at least one root $\grg_0\in S$ such that $\langle \grb , \grg_0^\vee \rangle < 0$. In particular $\grg_0+\grb$ is a root.

Suppose that $\grb < 0$. Then $\grb \in \widehat\Psi(v)$, hence $\grb \in \widehat\Psi(w)$ thanks to Proposition \ref{prop:weak-iff-strong}. 
Since $w \in \Wab$, this yields a contradiction because $\grg_0 + \grb$ cannot be a root by \eqref{eq:facile}.

Therefore $\grb > 0$.  Thus $\grb + \grg_0$ is either zero or a root. If it is zero, then $\grb\in -\widehat\Psi(w)$ proving our claim, so we can assume that it is a root.
Since $\gra \in \widehat \grD$ and $v(\grg_0) \in \widehat \Phi^-$, it follows that $v(\grb + \grg_0) = \gra + v(\grg_0) < 0$.

If $\grb + \grg_0 > 0$, then $\grb + \grg_0 \in -\widehat\Psi(v)$, thus $\grb+\grg_0 \in -\widehat\Psi(w)$ as well, 
and $w(\grb) = w(\grb+\grg_0) - w(\grg_0) \in \widehat \Phi^-$ proving $\grb \in -\widehat\Psi(w)$.

Otherwise we have $\grb + \grg_0 < 0$. Since $\grg_0 \in \widehat \Psi$ and $\grb > 0$, it follows that either 
$\grb + \grg_0 \in \Phi^-$ or $\grb \in \Phi^+$. If $\grb + \grg_0 \in \Phi^-$, then $w(\grb + \grg_0)<0$, hence 
$w(\grb) = w(\grb+\grg_0) - w(\grg_0) < 0$, and again we get $\grb \in -\widehat\Psi(w)$.

Suppose finally that $\grb \in \Phi^+$. Then it must be $\grb \in \grD$: if indeed $\grb = \grb_1 + \grb_2$ for some 
$\grb_1, \grb_2 \in \Phi^+$, then $\gra = v(\grb_1) + v(\grb_2)$, which is absurd since $\gra$ is simple and $v(\Phi^+) \subset \widehat \Phi^+$.

This proves the first claim. Suppose now that $\grb \in \grD$. Since $v\grs_S(\grb)=\grs_{v(S)}(\gra) < 0$, it must be either 
$\grs_S(\grb) \in -\widehat\Psi(v)$ or $\grs_S(\grb) < 0$. Suppose that $\grs_S(\grb) \in -\widehat\Psi(v)$ and write 
$\grs_S(\grb) = \grb - \sum_{\grg \in S} \langle \grb, \grg^\vee \rangle \grg$. Then 
$\langle \grs_S(\grb),\grg_0^\vee\rangle = -\langle \grb ,\grg_0^\vee \rangle >0$, and because 
$\grs_S(\grb) > 0$ it follows that $\grs_S(\grb)-\grg_0$ is a root, against the fact that both 
$\grs_S(\grb)$ and $-\grg_0$ are in $-\widehat\Psi(w)$ and $w\in\Wab$. Therefore $\grs_S(\grb)<0$, and the second claim is proved.
\end{proof}

Thus we give the following definition.

\begin{definition}\label{def:discesefiniteaffini}
Let $(v,S)$ be an admissible pair and let $\gra\in\widehat\Delta$ be a descent for $(v,S)$. We say that $\gra$ is of 
\textit{finite type} if $v^{-1}(\gra)\in\Delta$, and we say that it is of \textit{affine type} if $v^{-1}(\gra)\in -\widehat\Psi$.
\end{definition}

\begin{remark}
Let $\gra \in \widehat \grD$ be a descent for an admissible pair $(v,S)$ and denote $\grb = v^{-1}(\gra)$. 
By the previous Proposition we have $s_\gra v>v$. Moreover, if $\gra$ is affine then 
$\Psi(s_\gra v)=\Psi(v)\cup\{-\grb\}$ and in particular $s_\gra v$ is a minuscule element, while if $\gra$ is finite 
the minimal lenght coset representative of $s_\gra v $ in $\widehat W /W$ is equal to $v$ and $\grb$ is a descent for $\grs_S$, 
and it is real if and only if $\gra$ is a real descent for $\grs_{v(S)}$.
\end{remark}

We now characterize the affine descents of an admissible pair $(v,S)$ 
among the simple roots $\gra \in \widehat \grD$ such that $v^{-1}(\gra) \in -\widehat \Psi$.

\begin{proposition}	\label{prop:discese-affini}
Let $w \in \Wab$, let $(v,S)$ be a $w$-admissible pair and let $\gra \in \widehat \grD$ be such that $\grb=v^{-1}(\gra) \in -\widehat\Psi(w)$. 
Then $\gra$ is a descent for $(v,S)$ if and only if $\grb$ is not orthogonal to $S$, and it is real if and only if $\grb \in -S$.
\end{proposition}

\begin{proof}
If $\grb$ is orthogonal to $S$, then $\gra$ is orthogonal to $v(S)$ and $\grs_{v(S)}(\gra) = \gra$. Therefore, if $\gra$ is a descent, $\grb$ is not orthogonal to $S$. 
Suppose conversely that $\grb \in -\widehat\Psi(w)$ is not orthogonal to $S$. 
Since $\grb - \grg$ is never a root for $\grg \in S$, we have $\langle \grb, \grg^\vee \rangle \leq 0$ 
for all $\grg \in S$. On the other hand $v(\grg) <0$ for all $\grg \in S$, thus by \eqref{eq:sigma-vSa} 
we see that $\gra$ is a descent for $(v,S)$.

We now show that $\gra$ is real if and only if $\grb \in -S$. Clearly, $\gra$ is real if $\grb \in - S$.
Suppose conversely that $\gra$ is real, that is $\grs_S(\grb)=-\grb$. Since $\grb\in -\widehat\Psi$, by Proposition \ref{prop:radici-attive-reali}
we have then $\grb= -(\grg_1+\grg_2)/2$ and $\gra=-(v(\grg_1)+v(\grg_2))/2$, with $\grg_1,\grg_2\in S$. Since $\gra$ is simple and $v(\grg_1)$, $v(\grg_2)$ are both negative, 
these roots must be equal, thus $\grb=-\grg_1\in -S$. 
\end{proof} 

\section{The extended action, and dimension formulas}\label{sez:dimensione}

For $v\in \widehat W_{\ab}$, let $B_v$ be the split extension of $B$ defined by $B_v = B \ltimes \wgoa_v$, with product given by
\begin{equation}\label{eq:gruppo}
(p,x)(q,y) = (pq, \Ad_{q^{-1}}(x)+y).
\end{equation}
If $w \in \widehat W_\ab$ and $v\leq w$, then $B_v$ acts on $\wgoa_w$ by
\begin{equation}\label{eq:azione}
(p,x). y = \Ad_p(x+y). 
\end{equation}

Notice that $B = B_e$ (where $e \in W$ denotes the neutral element), and that $\widehat \aa_w$ is homogeneous under $B_w$. Moreover, by Proposition \ref{prop:weak-iff-strong} we have $B_v \subset B_{v'}$ if and only if $v \leq v'$: thus we can study the $B$-action on $\widehat \aa_w$ by studying inductively the action of $B_v$ on $\widehat \aa_w$, where $v \leq w$. Proceeding in this way, in this section we will show that $B$ acts with finitely many orbits on $\widehat \aa_w$, and we will give a parametrisation of the $B_v$-orbits in $\widehat \aa_w$ in terms of the $w$-admissible pairs introduced in the previous section, as well as a formula for their dimension. On the other hand, $\widehat \goa_w$ and $\goa_w$ are isomorphic $B$-modules, thus the previous discussion applies to $\goa_w$ as well.

Recall from the Introduction that, if $S$ is an orthogonal set of real roots, then $e_S=\sum_{\grg\in S}e_\grg$. Moreover, we define $f_S=e_{-S}$, and we set the convention
that $e_\varnothing=f_\varnothing=0$.
Since $e_{\pm\gra}$ and $e_{\pm\grb}$ commute for all $\gra,\grb\in S$, we have
$$
\grs_S=\exp{(-f_S)}\exp{(e_S)}\exp{(-f_S)}.
$$
Let $\widehat \calF=\widehat G/\widehat B$ be the affine flag variety, where $\widehat B \subset \widehat G$ is the Iwahori subgroup defined by $\widehat \Phi^+$, and let $x_0 \in \widehat \calF$ be the base point defined by $\widehat B$. 

\begin{lemma}\label{lem:eS}
Let $v\in \widehat W_{\ab}$, then $vBv^{-1}\subset \widehat B$ and $\Ad_v\wgoa_v\subset \widehat \gob$. If moreover $(v,S)$ is an admissible pair we have
$$ v \exp{(B_v e_S)}\, v^{-1} x_0 \subset \widehat B \grs_{v(S)} x_0.$$
In particular, if $(v,S)$ and $(v,S')$ are admissible pairs and $B_v e_S = B_v e_{S'}$, then $S=S'$. 
\end{lemma}

\begin{proof}
The first claim follows immediately from the inclusion $v(\Phi^+)\subset\Dap$ together with the definition of $\widehat\Psi(v)$.
To prove the second claim, notice that
$$
 v \exp{(B_ve_S)}\, v^{-1}  \subset  v\, B\, \exp(e_S)\cdot \exp(\wgoa_v) \, B\,v^{-1}.
$$
On the other hand, the right hand side of the previous equation equals $$vBv^{-1} \, \exp(e_{v(S)}) \, \exp(\Ad_v \wgoa_v) vBv^{-1},$$
thus by the first claim of the lemma together with the fact that $v(\grb)<0$ for all $\grb \in S$ we get 
$$
v \exp{(B_ve_S)}\, v^{-1} x_0 \subset \widehat B \exp(-f_{v(S)})\exp(e_{v(S)})\exp(-f_{v(S)}) x_0=\widehat B \grs_{v(S)}x_0.
$$
Suppose now  that $(v,S)$ and $(v,S')$ are admissible pairs such that $B_v e_S=B_v e_{S'}$. Then by the previous equality we get 
$\widehat B \grs_{v(S)} x_0= \widehat B \grs _{v(S')}x_0$, hence $\grs_{v(S)}=\grs_{v(S')}$, and $S=S'$ by Proposition \ref{prop:iniettivita-involuzioni}.
\end{proof}

\begin{proposition}\label{prop:decomposizione-ind}
Let $v \in \widehat W_{\ab}$ and let $\gra \in \widehat \grD$ be such that $s_\gra v \in \widehat W_{\ab}$ and  $v < s_\gra v$. Denote $\grb = -v^{-1}(\a)$ and suppose that $(s_\gra v,S)$ is an admissible pair, then the following holds.
\begin{enumerate}[\noindent i)]
     \item If $\grb$ is not orthogonal to $S$, then $(v,S)$ is admissible and
     	$$B_{s_\gra v} e_{S} = 	B_{v} e_{ S}.$$
           Moreover, $\grs(s_\gra v,S)=s_\gra\circ \grs(v,S)<\grs(v,S)$ and  $L(s_\gra v,S)=L(v,S)-1$.
 
     \item If $\grb$ is orthogonal to $S$, then $(v,S)$ and  $(v, S\cup\{\grb\})$ are both admissible and
  	   $$B_{s_\gra v} e_{S}=B_{v} e_{S} \sqcup B_{v} e_{S \cup \{\grb\}}.  $$
	   Moreover, $e_{ S} \in \overline{B_{v} e_{S\cup \{ \grb\}}}$ and
	   $$\dim B_{s_\gra v} e_{ S} = \dim B_{v} e_{ S\cup \{\grb\}} = \dim B_{v} e_{ S} +1.$$
	   Finally,
	   $\grs(s_\gra v,S)=\grs(v,S)=s_\gra\circ \grs(v,S\cup\{\grb\})< \grs(v,S\cup\{\grb\})$ and   
	   $L(s_\gra v,S)=L(v,S)=L(v,S\cup\{\grb\})-1$.
\end{enumerate}
\end{proposition}

\begin{proof}
Let $w \in \widehat W_\ab$ be such that $v \leq w$ and $S\subset \widehat\Psi(w)\smallsetminus \widehat\Psi(s_\gra v)$. Denote $v'=s_\gra v$ and $S'=S\cup \grb$, and notice that $\widehat\Psi(v') = \widehat\Psi(v) \sqcup \{\grb\}$. Thus $\wgoa_{v'} = \wgoa_{v} \oplus \gog_{\grb}$ and 
\begin{equation}	\label{eq:Bv-orbite}
	B_{v'} e_{ S} = B_{v} e_{ S} \cup B_{v} (e_{ S} + \mk^* e_{ \grb}).
\end{equation}

i) Suppose that $\grb$ is not orthogonal to $S$. Let $\grg_0 \in S$ be such that $\langle \grb , \grg_0^\vee \rangle 
\neq 0$. Since the sum of two elements in $\widehat\Psi(w)$ is never a root, it follows that $\langle \grb , \grg_0^\vee \rangle 
> 0$, hence $\grb - \grg_0 \in \Phi$. Denote $\gre = \grb - \grg_0$.
By Proposition \ref{prop:inversioni minimali e massimali}, the root $\grb$ is maximal in $\widehat\Psi(w)\smallsetminus \widehat\Psi(v)$, thus $\gre \in \Phi^+$ and by Lemma \ref{lemma:adding-roots} $\grg_0$ is the unique root  in $S$ 
such that $\grg_0+ \gre$ is a root. Hence
$$
u_\gre(t) (e_S +\wgoa_{v}) = e_S + t e_\grb +\wgoa_{v},
$$
and it follows $B_{v'} e_{ S} = B_{v} e_{ S}$.To show the last claim, since $v(\grg)<0$ and $\langle \grb , \grg^\vee \rangle \geq 0$ for all $\grg\in S$, notice that
$$\grs_{v(S)}(\gra) =  \gra + \langle \grb,\grg_0^\vee \rangle v(\grg_0) 
+ \sum_{\grg\neq \grg_0, \grg\in S}\langle \grb,\grg^\vee \rangle v(\grg) < \gra.
$$
Since $\gra$ is simple, it follows that $\gra$ is a descent for $\grs(v,S)$, and by Proposition \ref{prop:discese-affini} it is a complex descent. Thus $s_\gra \circ \grs_{v(S)} = s_\gra \grs_{v(S)} s_\gra = \grs_{v'(S)}$ and we conclude by Lemma
\ref{lemma:discese-CR}.

ii) Suppose now that $\grb$ is orthogonal to $S$. Then 
$(v,S)$ and $(v,S')$ are both admissible, and by Lemma \ref{lem:eS} the orbits  $B_v e_S$ and $B_v e_{S'}$ are different.
Since $S$ is orthogonal, we have $e_S+\mk^*e_\grb \subset T\cdot e_{S'}$, 
thus $B_{v}e_{S'}=B_{v} (e_{ S} + \mk^* e_{ \grb})$ and $e_S\in \overline{B_{v}e_{S'}}$. 
It follows from \eqref{eq:Bv-orbite} that
$B_{v'} e_S = B_v e_S \sqcup B_v e_{S'}$, and that $B_v e_{S'}$ is dense in $B_{v'} e_S$. In particular, $B_v e_{S'}$ has the same dimension of $B_{v'}e_S$.
On the other hand $\dim B_{v'} =\dim B_{v}+1$ and $\dim B_{v'} e_S$ is strictly bigger than $\dim B_{v}e_{S}$, which yields the last dimension equality. The last staments about the involutions follow trivially, since 
$s_\gra$ and $\grs_{v(S)}$ commute  and $\grs_{v(S\cup \{\grb\})}(\gra)=-\gra$. 
\end{proof}

We can now classify the $B_v$-orbits in the $B_v$-stable abelian subalgebras of $\gou$, and compute their dimensions.

\begin{theorem}	\label{teo:parametrization}
Let $v,w \in \Wab$ with $v \leq w$. Then there is a bijection between the  orthogonal subsets of $\widehat\Psi(w) \smallsetminus \widehat\Psi(v)$ 
and the $B_v$-orbits on $\wgoa_w$ given by $S \mapsto B_v e_S$. Moreover
$$
\dim B_ve_S=\ell(v)+L(v,S).
$$
\end{theorem} 

\begin{proof}
We have already proved in Lemma \ref{lem:eS} that $B_v$-orbits defined by different orhogonal subsets are different.
We show the remaining claims by decreasing induction on $\ell(v)$. If $v=w$, then $\wgoa_w = B_v e_\varnothing$ and both statements hold trivially since 
$\dim \wgoa_w=\ell(w)$.

Suppose that $\ell(v) < \ell(w)$, and let $\gra \in \widehat \grD$ be such that $s_\gra v \in \Wab$ and $v < s_\gra v \leq w$ 
(notice that such a simple root exists by Proposition \ref{prop:weak-iff-strong}). Denote $\grb = -v^{-1}(\gra)$, 
then $\wPsi(s_\gra v) =\wPsi(v)\sqcup \{\grb\}\subset \widehat \Psi(w)$. By the inductive assumption, every 
$B_{s_\gra v}$-orbit in $\wgoa_w$ has a base point of the form $e_{S}$, 
with $S \subset \wPsi(w) \smallsetminus \wPsi(s_\gra v)$ orthogonal. 
On the other hand, given such a subset $S$, by Proposition \ref{prop:decomposizione-ind} we have
$$
B_{s_\gra v} e_{S} =
\left\{	\begin{array}{cc}
B_v e_{S} \sqcup B_v e_{S \cup \{\grb\}} & \text{if $\grb \perp S$} \\ 
		B_v e_{S} & \text{if $\grb \not \perp S$}
		\end{array} \right. ,
$$
thus every $B_v$-orbit in $\wgoa_w$ has a base point of the form  $e_{S'}$ with $S' \subset \widehat \Psi(w) \smallsetminus \widehat \Psi(v)$ orthogonal.

To show the dimension formula, assume first that $\grb \not \perp S$. Then by the inductive assumption together with Proposition \ref{prop:decomposizione-ind} we have 
\begin{align*}
\dim B_ve_S&=\dim B_{s_\gra v}e_S= \ell(s_\gra v)+L(s_\gra v,S)\\& = 1+\ell(v)+L(v,S)-1=\ell(v)+L(v,S).
\end{align*}
Similar computations for the orbits $B_ve_S$ and $B_v e_{S\cup\{\grb\}}$ show the dimension formula in the case $\grb\perp S$.
\end{proof}

As a corollary we get the classification of the $B$-orbits in the $B$-stable abelian subalgebras of $\u$, and a formula for their dimension. A different proof of the first statement was already given by Panyushev \cite[Theorem 2.2]{Pa5}.

\begin{corollary}\label{cor:parametrizzazionedimensione}
Let $w \in \Wab$. Then there is a bijection between the orthogonal subsets of $\widehat\Psi(w)$ 
and the $B$-orbits on $\wgoa_w$ given by $S \mapsto B e_S$. Moreover
$$
\dim B e_S=L(\grs_S).
$$
\end{corollary}

\section{The Bruhat order of the $B_v$-orbits on $\widehat{\mathfrak a}_w$}\label{sez:principale}


In this section we will study the Bruhat order among the $B_v$-orbits in $\widehat\goa_w$, where $v,w \in \widehat W_\ab$ and $v \leq w$.
The description of the Bruhat order will be based on Lemma \ref{lem:eS} and on the description of some basic 
relations among orbit closures, which will allow us to proceed by induction.
Such basic relations are those described in Proposition \ref{prop:decomposizione-ind}, plus another one coming from the action of the minimal parabolic subgroups of $G$ which will be described in the following proposition.

Let $v\in\widehat W_{\ab}$ and let $\grb\in \Delta$ be such that $v(\grb)\in \widehat \Delta$. Then by Theorem \ref{teo:panyushev-normalizzatore} we have 
$s_\grb(\widehat\Psi(v))\subset \widehat\Psi(v)$, and the minimal parabolic subgroup $P_\grb\subset G$ normalizes $\wgoa_v$. 
In particular we can define a group $P_{\grb,v}=P_\grb \ltimes \wgoa_v$ by equation \eqref{eq:gruppo}, and if $v\leq w\in W_{\ab}$ then $P_{\grb,v}$ 
acts on $\wgoa_w$ by formula \eqref{eq:azione}. Notice that by construction $\dim P_{\grb,v}=\dim B_v +1$.

\begin{proposition} 	\label{prop:discese-finite}
Let $w \in \Wab$ and let $(v,S)$ be a $w$-admissible pair. Let $\gra \in \widehat \grD$ be a descent for $(v,S)$ of finite type and denote 
$\grb = v^{-1}(\gra)\in \Delta$. Then $\overline{B_v e_S} = \overline{P_{\grb,v} e_S}$, and the following hold.
\begin{enumerate}[\indent i)]
	\item If $\gra$ is complex, then $(v,s_\grb(S))$ is $w$-admissible and $e_{s_\grb(S)} \in P_{\grb,v} e_S$.
	\item If $\gra$ is real, write $\grb = \tfrac{1}{2}(\grg_1 - \grg_2)$ as in Proposition \ref{prop:radici-attive-reali} with 
	$\grg_1, \grg_2 \in S$ and set $S_\grb = (S \smallsetminus \{\grg_1, \grg_2\}) \cup \{\grb + \grg_2\}$. Then $(v,S_\grb)$ is $w$-admissible. 
	Moreover $\grs_{v(S_\grb)} = s_\gra\circ \grs_{v(S)}=s_\gra \grs_{v(S)}$ and $e_{S_\grb} \in P_{\grb,v} e_S$.
\end{enumerate}
\end{proposition}

\begin{proof}
Recall from Proposition \ref{prop:discese-AF} that $\grb$ is a descent for $\grs_S$, hence 
$s_\grb(S) \subset \widehat\Psi(w)$ by Proposition \ref{prop:sbetaammissible}.  Moreover it follows from Theorem \ref{teo:panyushev-normalizzatore} that $s_\grb(\widehat\Psi(v)) \subset \widehat\Psi(v)$, 
therefore $(v,s_\grb(S))$ is $w$-admissible. We distinguish two cases, depending on whether $\gra$ is complex or real.

i) Suppose that $\gra$ is a complex descent, then by Lemma \ref{lemma:discese-CR} we have
$\grs(v,S) > s_\gra \circ \grs(v,S) =\grs(v, s_\grb (S))$ and $L(v, s_\grb (S)) = L(v,S)-1$.
Hence $s_\grb(S) \neq S$ and 
$\dim B_v e_{s_\grb(S)}=\dim B_v e_S -1$ by 
Theorem \ref{teo:parametrization}. Notice that 
both $e_S$ and $e_{s_\grb(S)}$ belong to $P_{\grb,v} e_S$, therefore
$$
\dim B_v e_S \leq \dim P_{\grb,v} e_S \leq \dim B_v e_{s_\grb(S)} +1 =\dim B_v e_S
$$
and $B_v e_S$ is dense in $P_{\grb,v}e_S$.

ii) Suppose now that $\gra$ is a real descent for $(v,S)$. Since by definition $\grg_1 - \grb = \grg_2 + \grb$ and since $(v,S)$ is $w$-admissible, 
it follows that $(v,S_\grb)$ is $w$-admissible as well. Notice that $\{\grb, \grg_2\}$ is the base of a root subsystem of type $C_2$: 
thus $\grg_2 = s_\grb(\grg_1)$ and $s_{\grg_2}(\grb) = \grb+\grg_2$. 
It follows that $s_\grb s_{\grg_1} s_{\grg_2} = s_{\grb + \grg_2}$, 
hence $\grs_{v(S_\grb)} = s_\gra \grs_{v(S)}=s_\gra\circ \grs_{v(S)}$. 

We now prove that $e_{S_\grb} \in P_{\grb,v}e_S$. 
Let  $G_\grb\subset P_\grb$ be the subgroup generated by $U_\grb,U_{-\grb}$ and let $H$ be the subgroup generated by $G_\grb$, $U_{\grg_2}$, $U_{-\grg_2}$. Then $H$ 
is either isomorphic to $\mathrm{SO}_5$ or to $\mathrm{Sp}_4$, $G_\grb$ is isomorphic to $\mathrm{SL}_2$, and the action of $G_\grb$ on the space generated by 
$e_{\grg_2}, e_{\grg_2+\grb}, e_{\grg_1}$ is equivalent to the action of $\mathrm{SL}_2$ on the space of symmetric matrices of order 2 (namely to the adjont representation of $\mathrm{SL}_2$). In particular, we can normalise 
$e_{\grg_2}, e_{\grg_2+\grb}, e_{\grg_1}$ so that the action of an element $\left(\begin{smallmatrix}a&b\\c&d\end{smallmatrix}\right)$ is expressed in this basis by the matrix
$$
\begin{pmatrix}
a^2 & 2 ab & b^2 \\
ac & ad+bc & db \\
c^2 & 2dc & d^2
\end{pmatrix}
$$
Since $\Phi$ is not simply laced, the characteristic of $\mk$ is by assumption different from $2$. Let $i$ be a square root of $-1$ and let $g \in G_\grb$ the element represented by the matrix $\left(\begin{smallmatrix}
 1 & i \\
 i/2 & 1/2 
\end{smallmatrix}\right) \in \mathrm{SL}_2$. 
Then
$g \cdot (e_{\grg_1}+e_{\grg_2})=e_{\grg_2+\grb}$, hence $g\cdot e_S=e_{S_\grb}$ and
we can conclude as in the previous case.
\end{proof}


Let $(v,S)$ be an admissible pair, let $\gra \in \widehat \grD$ be a descent for $(v,S)$, and denote $\grb = v^{-1}(\gra)$. Then we define
\begin{equation}	\label{eqn:figli}
	\calF_\gra (v,S) = \left\{
	\begin{array}{lll}
		\big(s_\gra v,S\big) & & \text{if } \gra \text{ is complex of affine type,}\\	
		\big(v,s_\grb(S)\big) & & \text{if } \gra \text{ is complex of finite type,}\\
		\big(s_\gra v, S \smallsetminus \{-\grb\}\big) & & \text{if } \gra \text{ is real of affine type,}\\
				\big(v,S_\grb \big) & & \text{if } \gra \text{ is real of finite type.}
	\end{array}	\right.
\end{equation}

\begin{remark}	\label{oss:Falpha}
Denote $\calF_\gra(v,S)=(v',S')$, then by Propositions \ref{prop:decomposizione-ind} and \ref{prop:discese-finite} the following hold:
\begin{itemize}
 \item[i)] If $(v,S)$ is $w$-admissible for some $w \in \Wab$, then $\calF_\gra(v,S)$ is $w$-admissible.
 \item[ii)] $\grs(\calF_\gra(v,S))=s_\gra \circ \grs(v,S)<\grs(v,S)$ and $L(\calF_\gra(v,S))=L(v,S)-1$.
 \item[iii)] If $\gra$ is of finite type, then $v'=v,\,$ $e_S\in P_{\grb,v}e_{S'}$ and $\overline{P_{\grb,v}e_{S'}}=\overline{B_v e_S}$.
 \item[iv)] If $\gra$ is of affine type, then $v'=s_\gra v>v$ and $\overline{B_ve_S}=\overline{B_{v'}e_{S'}}$.
\end{itemize}
\end{remark}

We are now ready to prove our main theorem.

\begin{theorem}\label{teo:principale}
Let $v,w \in \widehat W_\ab$ be such that $v\leq w$, and 
let $(v,R)$ and $(v,S)$ be $w$-admissible pairs. Then $B_v e_R \subset \overline{B_v e_S}$ if and only if $\grs_{v(R)} \leq \grs_{v(S)}$.
\end{theorem}

\begin{proof}
Assume first that $e_R \in \overline{B_v e_S}$, then $v\exp(B_v e_R)v^{-1}\subset v\overline{\exp(B_v e_S)}v^{-1}$ and by Lemma 
\ref{lem:eS} we get that $\widehat B \grs_{v(R)}x_0\subset \overline{\widehat B\grs_{v(S)}x_0}$. Therefore the inequality $\grs_{v(R)}\leq \grs_{v(S)}$ follows from the description of the Bruhat order 
on the flag variety $\widehat G / \widehat B$.

To prove the other implication we proceed by induction both on $L(v,S)$ and on $\ell(\grs_{v(S)}) - \ell(\grs_{v(R)})$. 

Suppose that $\grs_{v(R)} = \grs_{v(S)}$ (in particular, this is the case if $L(v,S) = 0$): then $\grs_R = \grs_S$, and 
by Proposition \ref{prop:iniettivita-involuzioni} we get $R=S$. Therefore we can assume that $L(v,S) > 0$ and 
$\grs_{v(R)} < \grs_{v(S)}$. Let $\gra \in \widehat \grD$ be a descent for $(v,S)$ and denote $(v',S')=\calF_\gra(v,S)$.\\

\textit{Case 1.} Suppose that $\gra$ is a descent for $(v,R)$ as well. Notice that being  a descent of finite or affine type depends only on $v$ and $\gra$. In particular $\calF_\gra(v,R)$ and $\calF_\gra(v,S)$ share the same first component, and we can write $\calF_\gra(v,R) =(v',R')$ for some $R'$. By Remark \ref{oss:Falpha} ii) and by Lemma \ref{lemma:par} we get $\grs(\calF_\gra(v,R))\leq\grs(\calF_\gra(v,S))$, and by induction it follows
$B_{v'}e_{R'}\subset \overline{B_{v'}e_{S'}}$. 
If $\gra$ is of affine type, then we conclude by Remark \ref{oss:Falpha} iv).  Suppose that $\gra$ is of finite type. Then by \eqref{eqn:figli} we have $v' = v$, hence $B_v e_{R'}\subset \overline{B_v e_{S'}}$ and in particular $P_{\grb,v}e_{R'}\subset \overline{P_{\grb,v}e_{S'}}$. Thus we conclude by
Remark \ref{oss:Falpha} iii). \\

\textit{Case 2.} Suppose now that $\gra$ is a descent of finite type for $(v,S)$, and that it is not a descent for $(v,R)$. 
Then by \eqref{eqn:figli} we have $v'=v$, and by Remark \ref{oss:Falpha} ii) and  Lemma \ref{lemma:par} we have $\grs_{v(R)} \leq \grs_{v(S')}$. 
Notice that $\ell(\grs_{v(S')}) - \ell(\grs_{v(R)}) = \ell(\grs_{v(S)}) - \ell(\grs_{v(R)})-1$, thus we can argue by induction and by Remark \ref{oss:Falpha} iii) we obtain
$B_ve_R\subset \overline{B_ve_{S'}} \subset \overline{B_v e_S}$, which proves the claim.\\

\textit{Case 3.}
Suppose finally that $\gra$ is a descent of affine type for $(v,S)$, and that it is not a descent for $(v,R)$. Then by Remark \ref{oss:Falpha} iv) we have $v'=s_\gra v>v$. 
By Proposition \ref{prop:discese-affini} it follows that $\grb$ is orthogonal to $R$. Thus $(v', R)$ is $w$-admissible by Proposition \ref{prop:discese-AF}, and $\grs(v',R)=\grs(v,R)$. By Remark \ref{oss:Falpha} ii) together with Lemmas \ref{lemma:discese-CR} and \ref{lemma:par} we have $\grs(v',R)\leq \grs(v',S')$, thus we can 
argue by induction since $L(v',S')<L(v,S)$.
Then by Proposition \ref{prop:decomposizione-ind} and Remark \ref{oss:Falpha} iv) we get $B_v e_R \subset \overline{B_{v'} e_{R}} \subset \overline{B_{v'} e_{S'}} = \overline{B_{v} e_{S}}$, which proves our claim.
\end{proof}

\subsection{Further developments.}
We point out that the following stronger version of Theorem \ref{teo:principale} holds, without the need of fixing the minuscule element $w$.

\begin{theorem} \label{teo:forma-forte}
Let $v \in \widehat W_\ab$ and let $(v,R)$ and $(v,S)$ be admissible pairs, then $B_v e_R \subset \overline{B_v e_S}$ if and only if $\grs_{v(R)} \leq \grs_{v(S)}$.
\end{theorem}

Indeed, arguing by induction on $\ell(v)$ one can easily reduce the previous theorem to the case $v = e$, in which case we can prove the following theorem.

\begin{theorem} \label{teo:forma-fortissima}
Let $w \in \widehat W_\ab$ and let $S \subset \widehat \Psi(w)$ be orthogonal. Let $R \subset \widehat \Psi$ be an orthogonal subset such that $\grs_R \leq \grs_S$, then $R \subset \widehat \Psi(w)$ as well.
\end{theorem}

However we are able to prove Theorem \ref{teo:forma-fortissima} only involving case by case considerations based on the classification of the $B$-stable abelian subalgebras of $\gou$, and we do not have a uniform proof (see \cite[Proposition 2.13]{GMP} for a conceptual proof in a special case).

\vskip10pt
\footnotesize{

\noindent{\bf J.G.}: Dipartimento di Matematica, Universit\`a di Bologna, Piazza di Porta San Donato 5, 40126 Bologna, Italy;
{\tt jacopo.gandini@unibo.it}

\noindent{\bf A.M.}: Dipartimento di Matematica, Universit\`a di Pisa, Largo Bruno Pontecorvo 5, 56127 Pisa, Italy;
{\tt andrea.maffei@unipi.it}

\noindent{\bf P.M.F.}: Politecnico di Milano, Polo regionale di Como, 
Via Valleggio 11, 22100 Como, Italy;
{\tt pierluigi.moseneder@polimi.it}

\noindent{\bf P.P.}: Dipartimento di Matematica, Sapienza Universit\`a di Roma, Piazzale Aldo Moro 2, 00185 Roma, Italy;
{\tt papi@mat.uniroma1.it}
\end{document}